\documentclass[11pt,a4paper]{amsart}
\usepackage{amsmath,amsthm,amsfonts,amssymb,graphicx,psfrag}
\usepackage{color}

\usepackage{hyperref}
\usepackage[alphabetic,initials]{amsrefs}

\theoremstyle{plain}

\newtheorem{prop}{Proposition}[section]

\newtheorem*{mainthm}{Main theorem}

\newtheorem{cor}[prop]{Corollary}

\newtheorem{lemma}[prop]{Lemma}

\theoremstyle{definition}

\newtheorem{defn}[prop]{Definition}

\theoremstyle{remark}
\newtheorem{remark}[prop]{Remark}

\newcommand{\defin}{\textbf}

\newcommand{\dist}{\operatorname{dist}}
\newcommand{\dR}{{\operatorname{dR}}}
\newcommand{\End}{\operatorname{End}}
\newcommand{\ev}{\operatorname{ev}}

\newcommand{\GW}{\operatorname{GW}}

\newcommand{\Hom}{\operatorname{Hom}}

\newcommand{\LD}{L^{\otimes D}}

\newcommand{\mmax}{m_{{\operatorname{max}}}}

\newcommand{\PD}{\operatorname{PD}}

\newcommand{\pr}{\operatorname{pr}}
\newcommand{\pt}{\operatorname{pt}}

\newcommand{\st}{\operatorname{st}}

\newcommand{\Surj}{\operatorname{Surj}}

\newcommand{\CC}{{\mathbb C}}

\newcommand{\NN}{{\mathbb N}}
\newcommand{\QQ}{{\mathbb Q}}
\newcommand{\RR}{{\mathbb R}}
\newcommand{\TT}{{\mathbb T}}
\newcommand{\ZZ}{{\mathbb Z}}

\newcommand{\jJ}{{\mathcal J}}

\newcommand{\mM}{{\mathcal M}}

\newcommand{\uU}{{\mathcal U}}

\newcommand{\p}{\partial}
\renewcommand{\dbar}{\bar{\partial}}

\numberwithin{equation}{section}

\definecolor{blue}{rgb}{0,0,1}
\definecolor{red}{rgb}{1,0,0}
\definecolor{green}{rgb}{0,.7,0}

\title[Contact Hypersurfaces in Uniruled Manifolds Separate]{Contact Hypersurfaces in Uniruled Symplectic Manifolds Always Separate}
\author{Chris Wendl}
\address{Department of Mathematics \\ 
University College London \\
Gower Street \\
London WC1E 6BT \\ 
United Kingdom}
\email{c.wendl@ucl.ac.uk}
\thanks{Research supported by a Royal Society University Research Fellowship.}

\begin{document}

\begin{abstract}
We observe that nonzero Gromov-Witten invariants with marked point 
constraints in a
closed symplectic manifold imply restrictions on the homology classes that
can be represented by contact hypersurfaces.  As a special case,
contact hypersurfaces must always separate if the symplectic manifold is
uniruled.  This removes a superfluous assumption in a result of
G.~Lu \cite{Lu:uniruled}, thus implying that all contact manifolds that
embed as contact type hypersurfaces into
uniruled symplectic manifolds satisfy the Weinstein conjecture.
We prove the main result using the Cieliebak-Mohnke approach to
defining Gromov-Witten invariants via Donaldson hypersurfaces, 
thus no semipositivity or virtual moduli cycles are required.
\end{abstract}

\maketitle

\tableofcontents

\section{The statement}
\label{sec:statement}

\subsection{Main result and consequences}
\label{sec:main}

In this note, we prove the following.

\begin{mainthm}
Suppose $(M,\omega)$ is a closed symplectic manifold and
$V \subset M$ is a real hypersurface that is pseudoconvex for some choice
of $\omega$-compatible almost complex structure on~$M$.
Then the rational Gromov-Witten invariants of $(M,\omega)$, defined
in the sense of \cite{CieliebakMohnke:transversality}
(see \S\ref{sec:semipositive} and \S\ref{sec:CM}), satisfy
$$
\GW^{(M,\omega)}_{0,m,A}(\PD[V] \cup \alpha_1,\alpha_2,\ldots,\alpha_m ; \beta)
= 0
$$
for all $m \ge 3$, $A \in H_2(M)$, $\alpha_1,\ldots,\alpha_m \in H^*(M ; \QQ)$
and $\beta \in H_*(\overline{\mM}_{0,m} ; \QQ)$.
\end{mainthm}

Recall that a real hypersurface $V$ in an almost complex manifold
$(M,J)$ is \defin{pseudoconvex} (also sometimes called \defin{$J$-convex})
if the maximal $J$-invariant subbundle $\xi \subset TV$ is a contact structure
whose canonical conformal class of symplectic structures tames~$J|_{\xi}$.
As an important special case, when $(M,\omega)$ is a symplectic manifold,
we say $V \subset M$ is a \defin{contact type hypersurface} if
$\omega$ can be written in a neighborhood of $V$ as $d\lambda$ for some
$1$-form $\lambda$ whose restriction to $V$ is a contact form.  In that
case, $V$ is $J$-convex for any choice of $\omega$-tame almost complex
structure $J$ that preserves the contact structure on~$V$, and without loss of
generality one can also arrange $J$ to be $\omega$-compatible.

We will show in \S\ref{sec:GW} below that the main theorem has the 
following immediate consequence:

\begin{cor}
\label{cor:uniruled}
Suppose $(M,\omega)$ is a closed symplectic manifold that is
symplectically uniruled (see Definition~\ref{defn:uniruled}). 
Then every contact type hypersurface in $(M,\omega)$ is separating.
\end{cor}

Some motivation to prove such a result comes from the Weinstein conjecture, 
which asserts that any closed contact type hypersurface in a symplectic 
manifold has a closed orbit of its characteristic line field.  There is 
a long history of results that prove this conjecture under various assumptions 
on the existence of holomorphic curves in the ambient symplectic manifold, 
cf.~\cites{HoferViterbo:spheres,LiuTian,Lu:uniruled}.  However, such results 
have often been proved only for \emph{separating} contact hypersurfaces, 
leaving the question without this extra assumption open.  Our theorem 
thus shows that the extra assumption is superfluous, e.g.~combining it 
with Guangcun Lu's result, we obtain:

\begin{cor}[via \cite{Lu:uniruled}]
\label{cor:Weinstein}
If $(V,\xi)$ is a contact manifold that embeds into
a symplectically uniruled symplectic manifold as a contact type hypersurface, 
then every contact form for $(V,\xi)$ admits a periodic Reeb orbit,
i.e.~the Weinstein conjecture holds for~$(V,\xi)$.
\end{cor}

For more on symplectic manifolds to which this result applies, see 
\cite{Hyvrier:uniruled} and the references therein.

\begin{remark}
Our use of the technique of Cieliebak and Mohnke
\cite{CieliebakMohnke:transversality} for defining the Gromov-Witten
invariants via Donaldson hypersurfaces imposes certain technical
restrictions on the scope of the above results:
(1)~The setup in \cite{CieliebakMohnke:transversality} only handles
symplectic manifolds with integral cohomology, 
i.e.~$[\omega] \in H^2(M;\ZZ)$, due to the need for a 
symplectic hypersurface Poincar\'e dual to a large multiple of~$[\omega]$.
One can obviously generalize this to the assumption that $[\omega]$ is
any real multiple of an integral class, and of course every symplectic form
admits a small perturbation that has this property.
It is likely moreover that the restriction to integral classes 
can be lifted entirely by choosing symplectic hypersurfaces that
approximate the relevant homology classes, and indeed, the recent preprint
of Ionel and Parker \cite{IonelParker:virtual} claims to define
fully deformation-invariant Gromov-Witten invariants for arbitrary
$[\omega] \in H^2_\dR(M)$ using similar techniques.  For simplicity, we shall
nonetheless assume wherever necessary that $[\omega]$ is integral, in order
to remain fully consistent with \cite{CieliebakMohnke:transversality}.
(2)~Following \cite{MassotNiederkruegerWendl}, one can define a real 
hypersurface
$V$ in a symplectic manifold $(M,\omega)$ to be \defin{weakly contact} if
there exists an $\omega$-tame almost complex structure $J$ for which
$V$ is $J$-convex. This is equivalent to the condition required in our main
theorem if $\dim V = 3$, but in higher dimensions it appears to be more
general.  It is very likely that our main theorem holds under this weaker
assumption as well, and the proof given here will imply this at least in the 
semipositive case without coupling to gravity (using the standard setup
from \cite{McDuffSalamon:Jhol}).  A more general proof will probably be
possible in the future using polyfolds 
(cf.~Remark~\ref{remark:polyfolds}).  
In the non-semipositive case, our reliance on the Donaldson hypersurface
construction \cite{Donaldson:hypersurface} necessitates the added restriction 
that $J$ is \emph{compatible} with $\omega$, not just tamed.
\end{remark}

\subsection{Recollections on Gromov-Witten theory}
\label{sec:GW}

In this article, we regard the \emph{Gromov-Witten invariants} of a symplectic 
manifold $(M,\omega)$ as an association to each pair of integers 
$g,m \ge 0$ with $2g + m \ge 3$ and each homology class $A \in H_2(M)$ of a homomorphism
\begin{equation}
\label{eqn:GW}
\GW^{(M,\omega)}_{g,m,A} : H^*(M; \QQ)^{\otimes m} \otimes H_*( \overline{\mM}_{g,m} ; \QQ) \to \QQ,
\end{equation}
where $\overline{\mM}_{g,m}$ denotes the Deligne-Mumford compactification
of the moduli space of Riemann surfaces with genus~$g$ and $m$~marked points.
Let
$$
\PD : H_*(M;\QQ) \to H^*(M;\QQ)
$$
denote the Poincar\'e duality isomorphism, or its inverse when convenient.
In the absence of transversality problems, 
$\GW^{(M,\omega)}_{g,m,A}(\alpha_1,\ldots,\alpha_m ; \beta)$
is interpreted as a count of rigid unparametrized
$J$-holomorphic curves of genus~$g$, 
for a generic $\omega$-tame almost complex structure~$J$, with $m$ marked
points such that for $i=1,\ldots,m$, the $i$th marked point is mapped
to a generic smooth representative of $\PD(\alpha_i) \in H_*(M)$, 
and the underlying conformal structure of the 
domain lies in a generic smooth representative of $\beta \in 
H_*(\overline{\mM}_{g,m})$.
In practice, the transversality problems that
arise in this definition require considerable effort to overcome,
and the literature contains various approaches
(e.g.\ \cites{FukayaOno,LiTian,Ruan:virtual,
Siebert:GW,CieliebakMohnke:transversality,HWZ:GW}) which may or may not
all define the same invariants.  

In order to be concrete and also minimize
the technical apparatus needed, in this paper we shall work with the
definition provided by Cieliebak and Mohnke 
\cite{CieliebakMohnke:transversality} for the $g=0$ case, 
which uses a \emph{Donaldson
hypersurface} as auxiliary data and thus requires the symplectic form to
represent an integral cohomology class.  The essential details of this setup
will be reviewed in \S\ref{sec:CM}, though we shall also attempt to
express the main argument in terms that do not depend on these details.
In particular, the reader who would prefer to avoid serious technical issues
by assuming $(M,\omega)$ is semipositive may do so by skipping
from \S\ref{sec:semipositive} (where we review the main definitions in the
semipositive case) straight to \S\ref{sec:proof}.  In either case, the theory 
is defined essentially by constructing a
suitably compactified moduli space $\overline{\mM}_{0,m}^A(M,J)$ of 
stable nodal pseudoholomorphic spheres homologous to $A$, with $m$ marked points, 
such that the natural \emph{evaluation/forgetful map}
\begin{equation}
\label{eqn:smallpseudocycle}
(\ev,\Phi) = (\ev_1,\ldots,\ev_m,\Phi) : \mM_{0,m}^A(M,J) \to M^m \times
\overline{\mM}_{0,m}
\end{equation}
defines a \emph{rational pseudocycle} in the sense of 
\cite{McDuffSalamon:Jhol}*{\S 6.5}, meaning that rational intersection numbers
with homology classes in $M^m \times \overline{\mM}_{0,m}$ can be defined.
The homomorphism \eqref{eqn:GW} is then defined, up to a combinatorial
constant (see \eqref{eqn:GWgeneral}), by
\begin{equation}
\label{eqn:GWdefn}
\GW^{(M,\omega)}_{0,m,A}(\alpha_1,\ldots,\alpha_m;\beta) = 
[(\ev,\Phi)] \cdot \left( \PD(\alpha_1) \times \ldots \times
\PD(\alpha_m) \times \beta \right).
\end{equation}

\begin{remark}
\label{remark:gravity}
The Gromov-Witten invariants defined in \cite{CieliebakMohnke:transversality} do
not involve ``coupling to gravity,'' i.e.~they rely on the fact that
$\ev : \mM_{0,m}^A(M,J) \to M^m$ is a pseudocycle, but do not deal at all with the
forgetful map $\Phi : \mM_{0,m}^A(M,J) \to \overline{\mM}_{0,m}$, associating to a 
$J$-holomorphic curve its underlying conformal structure.  It is nonetheless true
in the context of \cite{CieliebakMohnke:transversality} that $(\ev,\Phi)$
is a pseudocycle and hence \eqref{eqn:GWdefn} is well defined; the proof of
this fact is almost already implicit in that paper, and we shall spell out the
missing ingredients in Appendix~\ref{sec:forgetful}.  Note that in the semipositive case, the
standard approach via domain-dependent almost complex structures suffices to
prove that the evaluation map is a pseudocycle, but not the forgetful 
map---see \cite{McDuffSalamon:Jhol}*{pp.~184--186}.  Thus the simplified
version of our arguments (avoiding Donaldson hypersurfaces) for the semipositive
case will be valid only for the simplified invariants
$\GW^{(M,\omega)}_{0,m,A} : H^*(M; \QQ)^{\otimes m} \to \ZZ$,
which match \eqref{eqn:GW} if $\beta$ is defined as the fundamental
class of $\overline{\mM}_{0,m}$.
\end{remark}

We now recall the following standard definition.

\begin{defn}\label{defn:uniruled}
A closed symplectic manifold $(M,\omega)$ is said to be
\defin{symplectically uniruled} if it has a nonzero rational Gromov-Witten
invariant with at least one pointwise constraint, i.e.~there exist
$A \in H_2(M)$, an integer $m \ge 3$ and classes
$\alpha_2,\ldots,\alpha_m \in H^*(M ; \QQ)$, 
$\beta \in H_*(\overline{\mM}_{0,m} ; \QQ)$ such that
\begin{equation}
\label{eqn:uniruled}
\GW^{(M,\omega)}_{0,m,A}(\PD[\pt],\alpha_2,\ldots,\alpha_m;\beta) \ne 0,
\end{equation}
where $[\pt] \in H_0(M)$ denotes the homology class of a point.
\end{defn}
Morally, being symplectically uniruled means one can find a set of constraints 
so that there is always a nonzero count of constrained holomorphic spheres 
passing through a generic point.

\begin{proof}[Proof of Corollary~\ref{cor:uniruled}]
If $V \subset M$ is a nonseparating hypersurface, then
$[V] \ne 0 \in H_*(M;\QQ)$ and one can therefore find a cohomology class
$\alpha_1 \in H^*(M;\QQ)$ with $\langle \alpha_1,[V] \rangle = 1$.
Hence
$$
\PD[V] \cup \alpha_1 = \PD[\pt].
$$
Now if $V$ is also pseudoconvex for some compatible almost complex structure,
then the main theorem implies that \eqref{eqn:uniruled} cannot be
satisfied for any choices $\alpha_2,\ldots,\alpha_m,\beta$, hence
$(M,\omega)$ is not uniruled.
\end{proof}

\begin{remark}
\label{remark:polyfolds}
An earlier version of the present paper made the optimistic claim that the 
arguments given here can be carried out using the polyfold theory of 
Hofer-Wysocki-Zehnder \cite{HWZ:GW}.  While that is
probably true, subsequent discussions with Hofer have led to the conclusion
that it is not fully provable using the technology in its present state: in
particular, homological intersection theory and Poincar\'e duality are
not currently well enough understood in the polyfold context to justify 
anything analogous to Equation~\eqref{eqn:Poincare}.  I would like to 
thank Joel Fish and Helmut Hofer for helping clarify this point.
\end{remark}

\subsection{Discussion}
\label{sec:discussion}

We now add a few more remarks on the context of the main theorem and its 
corollaries.

\subsubsection{Nonseparating hypersurfaces}

Nonseparating contact type hypersurfaces do exist in general, though they are
usually not easy to find.  A construction in dimension~$4$ was suggested by Etnyre 
and outlined in \cite{AlbersBramhamWendl}*{Example~1.3}: the idea is to start
from a symplectic filling with two boundary components, attach a Weinstein
$1$-handle to form the boundary connected sum and then attach a symplectic
cap to form a closed symplectic manifold, which contains both boundary
components of the original symplectic filling as nonseparating contact 
hypersurfaces.  At the time 
\cite{AlbersBramhamWendl} was written, examples of symplectic fillings with
disconnected boundary were known only up to dimension~6 (due to
McDuff \cite{McDuff:boundaries}, Geiges \cites{Geiges:disconnected,
Geiges:disconnected2} and Mitsumatsu \cite{Mitsumatsu:Anosov}), but
recently a construction in all dimensions appeared in work of the author
with Massot and Niederkr\"uger \cite{MassotNiederkruegerWendl}.
It seems likely that these examples can be combined with the symplectic capping
result of Lisca and Mati\'{c} \cite{LiscaMatic}*{Theorem~3.2} for Stein
fillable contact manifolds to construct examples of nonseparating contact
hypersurfaces in all dimensions, but we will not pursue this any further here.

Note that it is somewhat easier to find examples of \emph{weakly} contact hypersurfaces that
do not separate: for instance, considering the standard symplectic $\TT^4$ as a
product of two symplectic $2$-tori, for any nonseparating loop $\gamma \subset \TT^2$
the hypersurface $\gamma \times \TT^2 \subset \TT^4$ admits an obvious foliation
by symplectic $2$-tori, and this foliation can be perturbed to any of the
tight contact structures on~$\TT^3$ (cf.~\cite{Giroux:plusOuMoins}).  Notice that
one cannot use the same trick to produce a nonseparating weakly contact hypersurface
in $\TT^2 \times S^2$ with any product symplectic structure, as the latter is
uniruled.\footnote{Actually, the statement of our main theorem for $\TT^2 \times S^2$ 
can be proved by more elementary means without mentioning Gromov-Witten invariants, 
cf.~\cite{AlbersBramhamWendl}*{Theorem~1.15}.}  
This implies the well known fact (see~\cite{EliashbergThurston}) that the
obvious foliation by spheres on $S^1 \times S^2$ cannot be perturbed to a
contact structure.

\subsubsection{Higher genus}

The theorem of Lu \cite{Lu:uniruled} also establishes the Weinstein conjecture for
separating contact type hypersurfaces under the more general assumption
\begin{equation}
\label{eqn:GWgenus}
\GW^{(M,\omega)}_{g,m,A}(\PD([\pt]),\alpha_2,\ldots,\alpha_m ; \beta) \ne 0,
\end{equation}
i.e.~one need not assume $g = 0$.  In fact, using the more recent technology of
``stretching the neck'' \cite{SFTcompactness}, one can give a straightforward
alternative proof of Lu's result which also shows that any \emph{nonseparating}
contact hypersurface in a manifold satisfying \eqref{eqn:GWgenus} must have
a closed characteristic.\footnote{For this heuristic discussion we are ignoring
the usual analytical issues of how to define the higher genus Gromov-Witten 
invariants; definitions using the Donaldson hypersurface idea have appeared 
in recent work of Gerstenberger \cite{Gerstenberger:thesis} and
Ionel-Parker \cite{IonelParker:virtual}.}
Note however that in the genus zero case, this is
a weaker statement than Corollary~\ref{cor:Weinstein}: it asserts that a
\emph{particular} contact form on $(V,\xi) \subset (M,\omega)$ admits a
closed Reeb orbit, but not that this is true for every possible choice of contact form.
The obvious stretching argument does not appear to imply this stronger statement 
in general except when $V$ separates~$M$.

It seems unlikely moreover that our main result 
would hold under the more general assumption \eqref{eqn:GWgenus}---certainly 
the method of proof given below does not work, 
as it requires the fact that the relevant holomorphic curves in~$M$ can always be 
lifted to a cover (since $S^2$ is simply connected).  
However, it was pointed out
to me by Guangcun Lu that due to relations among Gromov-Witten invariants
(see \cite{Lu:pseudoCapacities}*{\S 7}),
certain conditions on higher genus invariants will imply that $(M,\omega)$ is also
uniruled, e.g.~this is the case whenever there is a nontrivial invariant of the form
$$
\GW^{(M,\omega)}_{g,m,A}(\PD([\pt]),\alpha_2,\ldots,\alpha_m ; [\pt]) \ne 0.
$$
The reason is that this invariant counts curves with a fixed conformal 
structure on the domain, so one can derive holomorphic spheres from them 
by degenerating the
conformal structure to ``pinch away'' the genus.

\begin{remark}
Note that in the above formulation of the Weinstein conjecture for closed 
contact hypersurfaces, the ambient symplectic manifold need not be closed, 
e.g.~every contact manifold is a contact hypersurface in its own (noncompact) 
symplectization.  As was shown in \cite{AlbersBramhamWendl}, there are many 
contact manifolds that do not admit any contact type embeddings into any closed 
symplectic manifold---as far as I am aware, all contact manifolds that are 
currently known to admit such embeddings are also symplectically fillable.
\end{remark}

\subsection{Acknowledgments}
I would like to thank Guangcun Lu for comments on a preliminary version of
this paper, Kai Cieliebak for feedback on the appendix, 
and Patrick Massot, Helmut Hofer, Joel Fish and
Jean-Paul Mohsen for useful conversations.
The question considered here was originally brought to my attention by a talk
of Cl\'ement Hyvrier about his paper \cite{Hyvrier:uniruled} at the
\textsl{Sixth Workshop on Symplectic Geometry, Contact Geometry and Interactions}
in Madrid, February 2--4, 2012, funded by the ESF's \textsl{CAST} programme.  
My approach to the proof owes a slight 
debt to an observation made by an anonymous referee for the paper \cite{AlbersBramhamWendl}.
Likewise, my understanding of Cieliebak-Mohnke transversality owes a
substantial debt to the CNRS-funded Summer School on Donaldson Hypersurfaces that took
place in La Llagonne, June 17--21, 2013.

\section{Some preparations}
\label{sec:preparations}

In this section, we shall review some crucial definitions, starting in
\S\ref{sec:pseudocycle} with the construction of the Gromov-Witten
pseudocycle in both the semipositive and general cases.
In \S\ref{sec:Mohsen}, we will also prove a simple result about Donaldson
hypersurfaces that is needed to carry out our application to contact
hypersurfaces in the non-semipositive case.

\subsection{Defining the Gromov-Witten pseudocycle}
\label{sec:pseudocycle}

We will now review the definitions of the moduli spaces that
determine the pseudocycle \eqref{eqn:smallpseudocycle}.  We begin with the
semipositive case in \S\ref{sec:semipositive} before addressing the
general case in \S\ref{sec:CM}.

\subsubsection{The semipositive case}
\label{sec:semipositive}

Recall that a closed $2n$-dimensional 
symplectic manifold $(M,\omega)$ is called
\defin{semipositive} if there are no spherical homology classes
$A \in \pi_2(M)$ satisfying
$$
\omega(A) > 0 \quad\text{ and } \quad 3 - n \le c_1(A) < 0.
$$
In particular, this is always satisfied if $n=2$ or~$3$.
Under this condition, one can define integer-valued Gromov-Witten 
invariants
$$
\GW^{(M,\omega)}_{0,m,A} : H^*(M; \QQ)^{\otimes m} \to \ZZ
$$
for any $m \ge 3$ and $A \in H_2(M)$ by the following prescription
explained in \cite{McDuffSalamon:Jhol}.  (The original construction of
these invariants is due to Ruan \cite{Ruan:GW}.)

Let $\jJ_\tau(M,\omega)$ denote the space of smooth $\omega$-tame almost
complex structures on~$M$, and define
$$
\jJ_{S^2} := \left\{ J \in \Gamma(\pr_2^*\End_\RR(TM))\ |\ 
\text{$J(z,\cdot) \in \jJ_\tau(M,\omega)$ for all $z \in S^2$} \right\},
$$
where $\pr_2 : S^2 \times M \to M$ denotes the projection.  We call
$\jJ_{S^2}$ the space of smooth $\omega$-tame \emph{domain-dependent}
almost complex structures (where the ``domain'' is $S^2$).
Given $J \in \jJ_{S^2}$, a smooth map $u : S^2 \to M$ is said to be
\defin{$J$-holomorphic} if for all $z \in S^2$,
\begin{equation}
\label{eqn:CR}
du(z) + J(z,u(z)) \circ du(z) \circ i = 0,
\end{equation}
where $i$ is the standard complex structure on~$S^2 = \CC \cup \{\infty\}$.
For any $m \ge 3$ and $A \in H_2(M)$, we can then define the moduli space
$$
\mM_{0,m}^A(M,J) = \left\{ (u, \mathbf{z}) \right\},
$$
where $u : S^2 \to M$ is a $J$-holomorphic map with $[u] = A$,
and $\mathbf{z} = (z_4,\ldots,z_m)$ is an ordered $(m-3)$-tuple of 
pairwise distinct points in
$S^2 \setminus \{0,1,\infty\}$.  Setting $(z_1,z_2,z_3) := (0,1,\infty)$,
the \defin{evaluation map} is then defined by
\begin{equation*}
\begin{split}
&\ev = (\ev_1,\ldots,\ev_m) : \mM_{0,m}^A(M,J) \to M^m, \\
&\ev_j(u,\mathbf{z}) = u(z_j) \quad \text{ for $j=1,\ldots,m$}.
\end{split}
\end{equation*}
The \defin{forgetful map} $\Phi : \mM_{0,m}^A(M,J) \to \mM_{0,m}$ is
likewise defined by associating to $(u,\mathbf{z})$ the equivalence
class of conformal structures on $S^2$ with $m$ marked points positioned
at $(0,1,\infty,z_4,\ldots,z_m)$.  Note that since we have fixed the 
positions of the first three marked points, there is no need to divide
out reparametrizations.

Under the semipositivity condition, 
one can show using standard index computations (see \cite{McDuffSalamon:Jhol})
that $\ev : \mM_{0,m}^A(M,J) \to M^m$ is a pseudocycle of dimension
$2(n-3) + 2 c_1(A) + 2m$ for generic choices of $J \in \jJ_{S^2}$, and
for such choices, the corresponding Gromov-Witten invariant 
(without coupling to gravity)
can be computed for $\alpha_1,\ldots,\alpha_m \in H^*(M;\ZZ)$ as
\begin{equation}
\label{eqn:GWsemipositive}
\GW^{(M,\omega)}_{0,m,A}(\alpha_1,\ldots,\alpha_m) = 
[\ev] \cdot \left( \PD(\alpha_1) \times \ldots \times \PD(\alpha_m) \right) \in \ZZ.
\end{equation}
As mentioned already in Remark~\ref{remark:gravity}, the forgetful map is
generally not a pseudocycle for this definition of the moduli space, and we
shall therefore ignore coupling to gravity in our discussion of the
semipositive case.

The genericity requirement in \eqref{eqn:GWsemipositive} implies that one cannot
generally assume $J$ to be domain-independent.  It will be important for
our application however that one can do the next best thing: fix any
$J_1 \in \jJ_\tau(M,\omega)$, which we shall refer to henceforward as the
\emph{reference} almost complex structure.  We can regard $J_1$ as an element
of $\jJ_{S^2}$ with constant dependence on $z \in S^2$, and the
tangent space at $J_1$ to the
Fr\'echet manifold $\jJ_{S^2}$ is then
\begin{equation*}
\begin{split}
T_{J_1}\jJ_{S^2} = \big\{ Y \in \Gamma(\pr_2^*\End_\RR(TM))\ |\ 
&\text{$Y(z,p) J_1(p) + J_1(p) Y(z,p) = 0$} \\
&\text{for all $(z,p) \in S^2 \times M$} \big\}.
\end{split}
\end{equation*}
After choosing a smooth family of metrics on the manifolds of
complex structures at points in~$M$, we can write any
$J \in \jJ_{S^2}$ in some $C^0$-small neighborhood of $J_1$ as
$J(z,p) = \exp_{J_1(p)} Y(z,p)$ for some $C^0$-small section
$Y \in T_{J_1} \jJ_{S^2}$.
Genericity then allows us to conclude the following:

\begin{lemma}
\label{lemma:sequenceSemipositive}
There exists a sequence $Y_k \in T_{J_1}\jJ_{S^2}$ converging to~$0$ in
$C^\infty$ such that \eqref{eqn:GWsemipositive} holds with the
Gromov-Witten pseudocycle $\ev : \mM_{0,m}^A(M,J) \to M^m$ defined 
for any $J = \exp_{J_1} Y_k$. \qed
\end{lemma}

\subsubsection{The Cieliebak-Mohnke approach}
\label{sec:CM}

We now consider $(M,\omega)$ to be an arbitrary closed $2n$-dimensional
symplectic manifold that satisfies $[\omega] \in H^2(M;\ZZ)$ but is 
not necessarily semipositive.  The purpose of this section is to
summarize the relevant details of the recipe from 
\cite{CieliebakMohnke:transversality} 
for defining the Gromov-Witten invariants.

As auxiliary data, we choose an $\omega$-compatible almost complex
structure $J_0$, and a so-called \emph{Donaldson hypersurface of degree
$D \in \NN$:}
$$
Z_D \subset (M,\omega) \text{ symplectic, such that }
\PD[Z_D] = D [\omega].
$$
The existence of $Z_D$ for large $D \gg 0$ is provided by a deep theorem of
Donaldson \cite{Donaldson:hypersurface}, and
we can assume moreover that $Z_D$ is \emph{nearly}
$J_0$-holomorphic, in the sense that its \emph{K\"ahler angle}
(see \cite{Donaldson:hypersurface}*{p.~669}) is arbitrarily small if
$D$ is sufficiently large.
It follows in particular that for any $\epsilon > 0$,
if $D > 0$ is sufficiently
large, one can find $J_1 \in \jJ_\tau(M,\omega)$ with
$\| J_1 - J_0 \|_{C^0} < \epsilon$ such that $Z_D$ is $J_1$-holomorphic.
We shall assume in the following that such a $J_1 \in \jJ_\tau(M,\omega)$
has been chosen and is fixed.

For an integer $k \ge 0$, suppose $T$ is a \defin{$k$-labelled tree}, i.e.~a
tree together with a partition of $\{1,\ldots,k\}$ assigning some
subset to each vertex $\alpha \in T$.  We shall write $\alpha E \beta$ whenever
$T$ contains an edge connecting the vertices $\alpha,\beta \in T$, and
denote by $\alpha_j \in T$ the vertex associated to $j \in \{1,\ldots,k\}$ by
the labelling.  Then if $S_\alpha$ denotes a copy of $S^2$ for each $\alpha \in T$,
we can regard a
\defin{nodal curve} with $k$ marked points \defin{modelled on $T$} as a tuple
$$
\mathbf{z} = \left( \{ z_{\alpha\beta} \in S_\alpha \}_{\alpha E \beta} , 
\{ z_j \in S_{\alpha_j} \}_{j \in \{1,\ldots,k\}} \right)
$$
such that for each $\alpha \in T$, all the points in this tuple lying on $S_\alpha$
(the \defin{special points}) are distinct.  We
associate to $\mathbf{z}$ the \emph{nodal Riemann surface}
$$
\Sigma_{\mathbf{z}} := \coprod_{\alpha \in T} S_\alpha \bigg/ z_{\alpha\beta} \sim
z_{\beta\alpha},
$$
where each component $S_\alpha$ is assumed to carry the standard complex structure~$i$.
The nodal curve $\mathbf{z}$ (or equivalently the nodal Riemann surface
$\Sigma_{\mathbf{z}}$) is called \defin{stable}
if for each vertex $\alpha \in T$, there are at least three special points; note that
this is actually a property of the labelled tree~$T$, so we can equivalently say $\mathbf{z}$ is
stable if it is modelled on a \defin{stable $k$-labelled tree}.
In this case, $\mathbf{z}$ represents an
element $[\mathbf{z}]$ of the Deligne-Mumford space~$\overline{\mM}_{0,k}$.  There is a natural
\defin{stabilization} map $\mathbf{z} \mapsto \st(\mathbf{z})$ 
that makes any nodal curve $\mathbf{z}$ into a
stable nodal curve $\st(\mathbf{z})$ by removing vertices with fewer than three special points
and placing marked points on neighboring vertices as necessary; this determines a
holomorphic surjection on the corresponding nodal Riemann surfaces
$$
\st : \Sigma_{\mathbf{z}} \to \Sigma_{\st(\mathbf{z})}.
$$

For each $\alpha \in T$, denote by $\jJ_{S_\alpha}$ a copy of the space $\jJ_{S^2}$ of
domain-dependent almost complex structures defined
in the previous section, and let
$$
\jJ_T := \prod_{\alpha \in T} \jJ_{S_\alpha}.
$$
For $J \in \jJ_T$, a \defin{nodal $J$-holomorphic map with $k$ marked points} is a pair 
$(\mathbf{z},\mathbf{u})$,
where $\mathbf{z}$ is a nodal curve with $k$ marked points modelled on $T$, and
$\mathbf{u} : \Sigma_{\mathbf{z}} \to M$ is a continuous map whose restriction to
each sphere $S_\alpha \subset \Sigma_{\mathbf{z}}$ is smooth and $J$-holomorphic 
(in the sense of \eqref{eqn:CR}) with
respect to the $S_\alpha$-dependent almost complex structure determined by~$J$.

Recall next that since $\overline{\mM}_{0,k+1}$ is a smooth manifold for any $k \ge 2$,
we can consider $\overline{\mM}_{0,k+1}$-dependent almost complex structures
$$
J \in \Gamma(\pr_2^*\End_\RR(TM)) \quad\text{ such that }\quad
J([\mathbf{z}],\cdot) \in \jJ_\tau(M,\omega),
$$
where as usual we denote the projection $\pr_2 : \overline{\mM}_{0,k+1} \times M \to M$.
For $k \ge 3$, this has a convenient interpretation using the canonical projection
$$
\pi : \overline{\mM}_{0,k+1} \to \overline{\mM}_{0,k}
$$
which forgets the last marked point and stabilizes the result.  Namely, for any
nodal curve $\mathbf{z}$ with $k$ marked points,
$\pi^{-1}([\st(\mathbf{z})])$ can be identified 
canonically with the nodal curve $\Sigma_{\st(\mathbf{z})}$, i.e.~we parametrize
$\pi^{-1}([\st(\mathbf{z})])$ via the position of the extra marked point.
Thus if $\mathbf{z}$ is modelled on the $k$-labelled
tree~$T$, we can associate to $\mathbf{z}$ and the family $J$ above a
$\Sigma_{\mathbf{z}}$-dependent almost complex structure
$$
J_{\mathbf{z}} \in \jJ_T, \qquad J_{\mathbf{z}}(z,\cdot) := J([\st(\mathbf{z}),\st(z)],\cdot),
$$
where we use $[\st(\mathbf{z}),\st(z)]$ as shorthand for the element of 
$\pi^{-1}([\st(\mathbf{z})]) \in
\overline{\mM}_{k+1}$ corresponding to $\st(z) \in \Sigma_{\st(\mathbf{z})}$ under the
above identification.  For technical reasons, it is important to consider
only families $J$ that are \defin{coherent} in the sense defined in
\cite{CieliebakMohnke:transversality}*{\S 3}, and we shall denote the space
of smooth $\overline{\mM}_{0,k+1}$-dependent $\omega$-tame almost complex structures
satisfying this condition by
$$
\jJ_{k+1} = \left\{ J : \overline{\mM}_{0,k+1} \to \jJ_\tau(M,\omega)\ |\ 
\text{$J$ is coherent} \right\}.
$$
For our purposes, all that we will need to know about the coherence
condition is stated in the following lemma, which follows immediately from
the definition in \cite{CieliebakMohnke:transversality}*{\S 3}.

\begin{lemma}
\label{lemma:coherence}
For any $J \in \jJ_{k+1}$, if $\mathbf{z}$ is a nodal curve modelled on the
$k$-labelled tree $T$, then for each $\alpha \in T$, the restriction of
the family 
$$
\Sigma_{\mathbf{z}} \to \jJ_\tau(M,\omega) : z \mapsto
J_{\mathbf{z}}(z,\cdot)
$$ 
to $S_\alpha$ depends only on $z \in S_\alpha$ and the 
special points of $\mathbf{z}$ on~$S_\alpha$. \qed
\end{lemma}

We can now define the moduli spaces needed for the Gromov-Witten invariants.
Given an integer $m \ge 0$ and $A \in H_2(M)$, let
$$
\ell := A \cdot [Z_D] = D \omega(A) \in \NN.
$$
We may easily assume $\ell > 3$ by making $D \in \NN$ sufficiently large (in general
it will be much larger).  Choose $J \in \jJ_{\ell+1}$ with the property that
$$
J([\mathbf{z}],\cdot) \equiv J_1 \quad \text{in a neighborhood of $Z_D$, for all
$[\mathbf{z}] \in \overline{\mM}_{0,\ell+1}$}.
$$
Using the canonical projection $\pi_m : \overline{\mM}_{0,m+\ell+1} \to
\overline{\mM}_{0,\ell+1}$ that forgets the first $m$ marked points and then
stabilizes, we can associate to $J$ a coherent 
$\overline{\mM}_{0,m+\ell+1}$-dependent almost complex structure $\pi_m^*J$.
Then for any nodal curve $\mathbf{z}$ modelled on an $(m+\ell)$-labelled 
tree~$T$, we regard a map $\mathbf{u} : \Sigma_{\mathbf{z}} \to M$ as 
$J$-holomorphic
if it satisfies the Cauchy-Riemann equation \eqref{eqn:CR} for the
$\Sigma_{\mathbf{z}}$-dependent almost complex structure
$(\pi_m^*J)_{\mathbf{z}}$.  Given homology classes
$$
\{A_\alpha \in H_2(M)\}_{\alpha \in T} \quad\text{ such that }\quad
\sum_{\alpha \in T} A_\alpha = A,
$$
the pair $(T,\{A_\alpha\})$ is called a \defin{weighted tree}, and 
it is called \defin{stable} if every vertex $\alpha \in T$ with $A_\alpha=0$
has at least three \emph{special points}, i.e.~marked points plus adjacent
vertices.  We define $\widetilde{\mM}_T^{\{A_\alpha\}}(M,J; Z_D)$ to be
the space of pairs $(\mathbf{z},\mathbf{u})$ as above such that
$\left[\mathbf{u}|_{S_\alpha}\right] = A_\alpha$ for each $\alpha \in T$ 
and $\mathbf{u}$ maps each of the last $\ell$ marked points into~$Z_D$.
Note that since $Z_D$ is $J$-holomorphic (as $J$ matches $J_1$ near~$Z_D$),
all isolated intersections of $\mathbf{u}$ with $Z_D$ are positive;
in particular, whenever $\mathbf{z}$ has no nodes and $A \ne 0$, the
relation $\ell = A \cdot [Z_D]$ implies that either the image of
$\mathbf{u}$ is contained in $Z_D$ or the intersections of $\mathbf{u}$
with $Z_D$ occur \emph{only} at the last $\ell$ marked points.  The
former is excluded under
suitable assumptions on $J$ and for sufficiently large $D \in \NN$, due to
\cite{CieliebakMohnke:transversality}*{Propositions~8.13 and~8.14}.

\begin{remark}
\label{remark:positive}
The class of holomorphic curves defined above has the
crucial property that \emph{all} isolated intersections with $Z_D$ are positive,
not only the guaranteed intersections at the last $\ell$ marked points.
Since the count of these intersections is controlled topologically, 
positivity provides 
the necessary lower bound on the number of marked points on components
of nodal curves, guaranteeing that such curves have stable domains
(see \cite{CieliebakMohnke:transversality} for details).
\end{remark}

We write $(\mathbf{z},\mathbf{u}) \sim (\mathbf{z}',\mathbf{u}')$ if
there exists a biholomorphic isomorphism
between the nodal curves $\mathbf{z}$ and $\mathbf{z}'$ such that
$\mathbf{u}$ and $\mathbf{u}'$ are correspondingly related by
reparametrization.  We then define the moduli space of 
\defin{$J$-holomorphic curves modelled on $(T,\{A_\alpha\})$} as
$$
\mM_T^{\{A_\alpha\}}(M,J; Z_D) = 
\widetilde{\mM}_T^{\{A_\alpha\}}(M,J; Z_D) \big/ \sim,
$$
along with the \defin{evaluation map},
$$
\ev = (\ev_1,\ldots,\ev_m) : \mM_T^{\{A_\alpha\}}(M,J; Z_D) \to M^m,
$$
which evaluates $\mathbf{u}$ at its first $m$ marked points.
If $m \ge 3$, we can also define the \defin{forgetful map}
$$
\Phi : \mM_T^{\{A_\alpha\}}(M,J; Z_D) \to \overline{\mM}_{0,m},
$$
which forgets both the map $\mathbf{u}$ and the last $\ell$ marked
points of $\mathbf{z}$, and then stabilizes the resulting nodal
curve with $m$ marked points.  The \emph{top stratum} is the component
$$
\mM_{0,m+\ell}^A(M,J; Z_D) := \mM_T^{\{A_\alpha\}}(M,J; Z_D),
\text{ where $|T| = 1$},
$$
consisting of equivalence classes $[(\mathbf{z},\mathbf{u})]$ such that
$\mathbf{z}$ has no nodes; in this case $\mathbf{u} : S^2 \to M$ is simply
a pseudoholomorphic sphere, for some domain-dependent almost complex
structure determined by $J$ and the positions of its last~$\ell$ marked points.
The union of the spaces
$\mM_T^{\{A_\alpha\}}(M,J; Z_D)$ for all stable weighted trees 
$(T,\{A_\alpha\})$ with $\sum_\alpha A_\alpha = A$ carries a 
natural topology as a metrizable Hausdorff space, the \emph{Gromov topology},
and we denote by
$$
\overline{\mM}_{0,m+\ell}^A(M,J; Z_D) \subset
\bigcup_{\text{$(T,\{A_\alpha\})$ stable}} \mM_T^{\{A_\alpha\}}(M,J; Z_D)
$$
the closure of $\mM_{0,m+\ell}^A(M,J; Z_D)$ in this space.

If $m \ge 3$, then for suitable choices of $J \in \jJ_{\ell+1}$ matching 
the reference structure $J_1$ near $Z_D$,
\begin{equation}
\label{eqn:pseudocycle}
(\ev,\Phi) : \mM_{0,m+\ell}^A(M,J;Z_D) \to M^m \times \overline{\mM}_{0,m}
\end{equation}
is a pseudocycle of dimension
$$
\dim \mM_{0,m+\ell}^A(M,J;Z_D) = 2(n-3) + 2c_1(A) + 2m,
$$
and the resulting rational Gromov-Witten invariants
$$
\GW^{(M,\omega)}_{0,m,A} : 
H^*(M; \QQ)^{\otimes m} \otimes H_*( \overline{\mM}_{0,m} ; \QQ) \to \QQ,
$$
\begin{equation}
\label{eqn:GWgeneral}
\begin{split}
\GW^{(M,\omega)}_{0,m,A}(\alpha_1,\ldots,\alpha_m,\beta) &= \\
\frac{1}{\ell!} [(\ev, & \Phi)] \cdot \left( \PD(\alpha_1) \times
\ldots \times \PD(\alpha_m) \times \beta \right)
\end{split}
\end{equation}
are independent of all choices.  If one excludes the forgetful map and
$\beta \in H_*\left(\overline{\mM}_{0,m}\right)$ from this statement, then it 
is simply the main result of \cite{CieliebakMohnke:transversality} (and is
also valid for any $m \ge 0$).  
We will explain in Appendix~\ref{sec:forgetful} how the arguments of
Cieliebak and Mohnke
can be modified to include the forgetful map in the discussion.

As alluded to above, the constructions in \cite{CieliebakMohnke:transversality}
require some extra assumptions on $J \in \jJ_{\ell+1}$ in order to define
the Gromov-Witten invariants, but the details of these assumptions will not
concern us beyond the following analogue of Lemma~\ref{lemma:sequenceSemipositive}.
Recall that we have fixed a reference almost complex structure $J_1$
for which the Donaldson hypersurface $Z_D$ is $J_1$-holomorphic.  We can
trivially regard $J_1$ as an element of $\jJ_{\ell+1}$ with constant
dependence on $\overline{\mM}_{0,\ell+1}$.  Then any other element of
$\jJ_{\ell+1}$ that is $C^0$-close to $J_1$ can be written as
$$
J = \exp_{J_1} Y
$$
for some $Y \in T_{J_1} \jJ_{\ell+1}$, where the latter is the Fr\'echet space of
\emph{coherent} (see \cite{CieliebakMohnke:transversality}*{\S 3}) smooth 
sections of $\pr_2^*\End_\RR(TM) \to \overline{\mM}_{0,\ell+1} \times M$ satisfying
$$
Y([\mathbf{z}],p) J_1(p) + J_1(p) Y([\mathbf{z}],p) = 0
\quad \text{ for all $([\mathbf{z}],p) \in \overline{\mM}_{0,\ell+1} \times M$}.
$$

\begin{lemma}
\label{lemma:sequence}
There exists a sequence $Y_k \in T_{J_1} \jJ_{\ell+1}$ converging to~$0$ in 
$C^\infty$ such that \eqref{eqn:GWgeneral} holds with the Gromov-Witten 
pseudocycle \eqref{eqn:pseudocycle} defined for any
$J = \exp_{J_1} Y_k$. \qed
\end{lemma}

\subsection{Donaldson hypersurfaces transverse to a contact hypersurface}
\label{sec:Mohsen}

In order to apply the Gromov-Witten invariants of 
\cite{CieliebakMohnke:transversality} to a situation involving 
pseudoconvex hypersurfaces, we need the following additional fact about
Donaldson hypersurfaces.

\begin{prop}
\label{prop:hypersurfaces}
Suppose $(M,\omega)$ is a closed $2n$-dimensional symplectic manifold with 
$[\omega] \in H^2(M;\ZZ)$, $J_0$ is an $\omega$-compatible
almost complex structure, and $V \subset M$ is a closed $(2n-1)$-dimensional 
$J_0$-convex hypersurface with induced contact structure
$$
\xi = TV \cap J_0(TV) \subset TV.
$$  
Then for all $D \in \NN$ sufficiently large, there exists a Donaldson hypersurface
$Z_D \subset (M,\omega)$ of degree $D$ 
that intersects $V$ transversely in a contact submanifold
of $(V,\xi)$.  Moreover, for any $\epsilon > 0$, if $D \in \NN$ is sufficiently
large, then one can find $Z_D$ with the above property and 
an $\omega$-tame almost complex structure $J_1$ on $M$ such that
\begin{enumerate}
\item $Z_D$ is $J_1$-holomorphic;
\item $V$ is $J_1$-convex with $\xi = TV \cap J_1(TV)$;
\item $\| J_1 - J_0 \|_{C^0} < \epsilon$.
\end{enumerate}
\end{prop}

The proposition is a straightforward application of Mohsen's 
relative version \cite{Mohsen:hypersurfaces} of an
estimated transversality result of Donaldson and Auroux
\cites{Donaldson:hypersurface,Auroux:asympHolomorphic}.
To explain this, we must recall some details from the
asymptotically holomorphic methods of Donaldson and Auroux,
as used by Mohsen.

We first need to define a quantitative measurement of the
distance of a real subspace of a complex vector space from
being complex.  Suppose $(E,J)$ is a finite-dimensional
complex vector space with Hermitian
inner product~$g$, and write $|v| := \sqrt{g(v,v)}$ for $v \in E$.  Then
for any real-linear subspace $E' \subset E$ of even dimension, define
\begin{equation*}
\begin{split}
\Theta_g(E' ; E,J) &:= \max_{v \in E',\ |v|=1}
\dist\left( J v, E'\right) \\
&= \max_{v \in E',\ |v|=1} \left( \min_{w \in E'}
| Jv - w| \right).
\end{split}
\end{equation*}
It will be useful to note that this definition depends on the Hermitian
metric only up to positive rescaling, i.e.
\begin{equation}
\label{eqn:conformal}
\Theta_{cg}(E' ; E,J) = \Theta_g(E' ; E,J) \quad \text{ for all $c > 0$}.
\end{equation}
It also depends
continuously on all the data, thus if $B$ is a compact space
and $(E,J) \to B$ is a complex vector bundle of finite rank
with Hermitian bundle metric~$g$,
then for any real subbundle $E' \subset E$ of even rank, we
can similarly define
$$
\Theta_g(E' ; E,J) := \max_{p \in B} \Theta_g(E'_p ; E_p,J) \ge 0.
$$
Observe that if $\omega$ is any symplectic structure on $(E,J)$ that tames~$J$,
then any sufficiently small perturbation of a complex subbundle is
automatically also a symplectic subbundle, thus we have the following.
\begin{lemma}
\label{lemma:Theta}
Suppose $B$ is a compact space and $(E,J) \to B$ is a complex vector 
bundle of finite rank, equipped with a Hermitian bundle metric~$g$.
In each of the following
statements, assume $E' \subset E$ is a real subbundle of even rank.
\begin{enumerate}
\item[(a)] $E'$ is a complex subbundle of 
$(E,J)$ if and only if $\Theta_g(E' ; E,J) = 0$.
\item[(b)] For any $C^0$-open neighborhood $\uU_J$ of $J$ in the space of
smooth complex structures on~$E$, 
there exists a number $c > 0$ such that
every $E' \subset E$ with $\Theta_g(E' ; E,J) < c$
is a complex subbundle of $(E,J')$ for some $J' \in \uU_J$.
\item[(c)] For any symplectic structure $\omega$ on $E \to B$ that tames~$J$,
there exists a number $c' > 0$ such that every $E' \subset E$ 
satisfying $\Theta_g(E' ; E,J) < c'$ is a symplectic subbundle of 
$(E,\omega)$.
\end{enumerate} \qed
\end{lemma}

In order to relate the above definition to questions of estimated 
transversality, we define (following \cite{Mohsen:hypersurfaces}) 
for any real-linear map $A : V \to W$ between finite-dimensional
Euclidean vector spaces, the \defin{surjectivity modulus}
$$
\Surj(A) := \min_{\lambda \in W^* \setminus \{0\}}
\frac{\| \lambda \circ A \|}{\|\lambda\|} \ge 0.
$$
\begin{lemma}
\label{lemma:surjectivity}
The surjectivity modulus has the following properties.
\begin{enumerate}
\item[(a)] $\Surj(A) > 0$ if and only if $A$ is surjective, and in this case
$$
\Surj(A) \ge \sup\left\{ \frac{1}{\| B \|} \ \bigg|\ 
\text{$B : W \to V$ is a right inverse of $A$} \right\}.
$$
\item[(b)]
For any two real-linear maps $A, B : V \to W$,
$$
\Surj(A + B) \ge \Surj(A) - \| B \|.
$$
\item[(c)]
Suppose $(V,J,g)$ and $(V',J',g')$ are finite-dimensional 
Hermitian vector spaces and
$A = A^{1,0} + A^{0,1} : V \to V'$ is real-linear, where $A^{1,0}$ and
$A^{0,1}$ denote the complex linear and antilinear parts respectively.
Then
\begin{equation}
\label{eqn:ratio}
\Theta_g(\ker A ; V,J) \le 2 \frac{\| A^{0,1} \|}{\Surj(A)}.
\end{equation}
\end{enumerate}
\end{lemma}
\begin{proof}
The first two properties are proved by straightforward computations.
The following proof of the third property was explained to me by
Jean-Paul Mohsen.

Let $V_{\ker A}^* = \left\{ \mu \in V^*\ |\ \mu|_{\ker A} = 0 \right\}$,
which is precisely the space of dual vectors on $V$ of the form
$\{ \mu = \lambda \circ A \in V^*\ |\ \lambda \in W^* \}$.
Now suppose $v \in \ker A$ and $|v| = 1$.  The distance of $Jv$ from
$\ker A$ is the norm of its second part under the orthogonal 
decomposition $V = (\ker A) \oplus (\ker A)^{\perp}$, hence
\begin{equation*}
\begin{split}
\dist(Jv,\ker A) &= \max_{w \in (\ker A)^{\perp} \setminus \{0\}} 
\frac{\left| \langle w,Jv \rangle \right|}{|w|} =
\max_{\mu \in V_{\ker A}^* \setminus \{0\}} \frac{|\mu(Jv)|}{\|\mu\|} \\
&= \max_{\lambda \in W^* \setminus \{0\}} \frac{|\lambda \circ A(Jv)|}{\|\lambda \circ A\|}.
\end{split}
\end{equation*}
Now, using the fact that $Av=0$ and that $A^{1,0}$ commutes while $A^{0,1}$
anticommutes with the complex structures, we have
$$
A(Jv) = A^{1,0}Jv + A^{0,1}Jv = J' A^{1,0}v - J' A^{0,1} v = 
-2 J' A^{0,1} v,
$$
hence $|A(Jv)| \le 2 \| A^{0,1} \|$, implying
$$
\dist(Jv,\ker A) \le \max_{\lambda \in W^* \setminus \{0\}}
\frac{2 \| \lambda \| \cdot \|A^{0,1}\|}{\| \lambda \circ A \|} =
2 \frac{\| A^{0,1} \|}{\Surj(A)}.
$$
\end{proof}

Next, assume $(M,\omega)$ is a closed symplectic manifold with
$[\omega] \in H^2(M;\ZZ)$, and $J_0$ is an $\omega$-compatible almost
complex structure.  This determines the sequence of Riemannian metrics
$$
g := \omega(\cdot,J\cdot), \qquad g_D := D \cdot g \text{ for $D \in \NN$}
$$
on~$M$.  Let $L \to M$ denote a complex line bundle with $c_1(L) = [\omega]$,
equipped with a Hermitian metric $\langle\ ,\ \rangle$ and
a Hermitian connection $\nabla$ whose curvature $2$-form is $-2\pi i \omega$.  
For $D \in \NN$, we also consider the $D$-fold tensor power $\LD \to M$, 
with its induced Hermitian metric and Hermitian
connection, also denoted by $\langle\ ,\ \rangle$ and $\nabla$ respectively; 
the latter has curvature $-2\pi i D\omega$.  For sections
$s : M \to \LD$, we denote by $\p s$ and $\dbar s$ respectively
the complex linear and antilinear parts of the covariant 
derivative~$\nabla s$.  We will always define $C^0$-norms of
$\nabla s$ and related tensors with respect to the metrics $g_D$ on
$TM$ and $\langle\ ,\ \rangle$ on $\LD$, e.g.
\begin{equation*}
\begin{split}
\| \nabla s(p) \|_{g_D} &:= \max_{X \in T_p M \setminus \{0\}}
\frac{| \nabla_X s |}{|X|_{g_D}} \qquad \text{ for $p \in M$},\\
\| \nabla s \|_{g_D} &:= \sup_{p \in M} \| \nabla s(p) \|_{g_D},
\end{split}
\end{equation*}
where $| X |_{g_D} := \sqrt{g_D(X,X)}$ for $X \in T_p X$ and
$| v | := \sqrt{\langle v,v \rangle}$ for $v \in \LD_p$.
The surjectivity modulus of $\nabla s(p)$ at points $p \in M$ will also
be defined relative to this choice of metrics, which we shall indicate
via the notation
$$
\Surj_{g_D}(\nabla s(p)) := \min_{0 \ne \lambda \in \Hom_\RR\left(\LD_p,\RR\right)}
\frac{ \| \lambda \circ \nabla s(p) \|_{g_D}}{\| \lambda \|}.
$$
This means $\Surj_{g_D}(\nabla s(p)) = \frac{1}{\sqrt{D}} \Surj_g(\nabla s(p))$.

The next two definitions are essentially due to Auroux 
\cite{Auroux:asympHolomorphic}, though we have made minor modifications
to fit them into the framework of \cite{Mohsen:hypersurfaces}.

\begin{defn}
Given constants $C > 0$ and $r \in \NN$, we say that 
a sequence of sections $s_D : M \to \LD$ (for large $D \in \NN$) is
\defin{$C$-asymptotically holomorphic up to order $r \in \NN$} if
for all $D$ sufficiently large,
\begin{equation}
\label{eqn:asympHol}
\begin{split}
\| s_D \|_{g_D} \le C, \qquad
&\| \nabla^m s_D \|_{g_D} \le C, \qquad
\| \nabla^{m-1} \dbar s_D \|_{g_D} \le \frac{C}{\sqrt{D}} \\
&\text{ for each $m=1,\ldots,r$}.
\end{split}
\end{equation}
\end{defn}

\begin{defn}
Given a constant $\eta > 0$ and a submanifold $V \subset M$, 
we say that a sequence of sections
$s_D : M \to \LD$ (for large $D \in \NN$) is \defin{$\eta$-transverse
along~$V$} if for all sufficiently large $D$,
$$
| s_D(p) | < \eta \quad \Rightarrow \quad 
\Surj_{g_D}\left( \nabla s_D(p)|_{T_p V} \right) \ge \eta \qquad
\text{ for all $p \in V$.}
$$
\end{defn}

For any $(M,\omega)$ and $J_0$ as above,
Donaldson \cite{Donaldson:hypersurface} constructs a sequence of
sections $s_D : M \to \LD$ that are, for some $K , \eta > 0$, 
$K$-asymptotically holomorphic 
up to order~$2$ and globally $\eta$-transverse 
(i.e.~$\eta$-transverse along~$M$).
It follows via \eqref{eqn:conformal} and 
Lemma~\ref{lemma:surjectivity}(c) that for sufficiently 
large $D \in \NN$, $Z_D := s_D^{-1}(0) \subset M$ are smooth submanifolds with
\begin{equation*}
\begin{split}
\Theta_g(TZ_D ; TM|_{Z_D},J_0) &= \Theta_{g_D}(TZ_D ; TM|_{Z_D},J_0) \\
&\le \max_{p \in Z_D}
\frac{2 \| \dbar s_D(p) \|_{g_D}}{\Surj_{g_D}\left( \nabla s(p) \right)} \\
&\le 2 \frac{ K / \sqrt{D}}{\eta} \to 0 \quad\text{ as $D \to \infty$}.
\end{split}
\end{equation*}
Thus by Lemma~\ref{lemma:Theta}, the submanifolds 
$Z_D \subset (M,\omega)$ are symplectic and uniformly close to being
$J_0$-holomorphic for sufficiently large~$D$.  These are the
Donaldson hypersurfaces that we made use of in the previous section; 
indeed, they satisfy $\PD[Z_D] = c_1(\LD) = D c_1(L) = D[\omega] \in H^2(M)$.

For our purposes, the relevant case of Mohsen's extension of the 
Donaldson-Auroux transversality theorem can now be stated as follows.

\begin{prop}[\cite{Mohsen:hypersurfaces}*{Th\'{e}or\`{e}me~2.2}]
\label{prop:Mohsen}
Assume $(M,\omega)$ is a closed $2n$-dimensional symplectic manifold with 
an $\omega$-compatible almost complex structure $J_0$, $V \subset M$ is a 
closed submanifold of dimension $2n-1$, and
$\xi \subset TV$ denotes the $J_0$-complex subbundle
$$
\xi := TV \cap J_0(TV).
$$
Then given any $K > 0$, $\epsilon > 0$ and $\mmax \in \NN$, there exist 
$D_0 \in \NN$  and $\eta > 0$ such that the following holds.
For any sequence of sections $s_D : M \to \LD$ (for large~$D$) which are
$K$-asymptotically holomorphic up to order~$2$, there exists a
sequence (for large~$D$) of sections $t_D : M \to \LD$ such that,
for all $D \ge D_0$, the sequence $t_D$ is
$\epsilon$-asymptotically holomorphic up to order~$\mmax$, and the
sequence $s_D' := s_D + t_D$ is $\eta$-transverse along~$V$, and
also satisfies
$$
\text{$p \in V$ and $|s_D'(p)| < \eta$} \quad\Rightarrow\quad
\Surj_{g_D}\left( \nabla s_D'(p)|_{\xi_p} \right) \ge \eta.
$$ \qed
\end{prop}

\begin{proof}[Proof of Proposition~\ref{prop:hypersurfaces}]
Assume $V \subset M$ is $J_0$-convex, and let $s_D : M \to \LD$ denote the
$K$-asymptotically holomorphic and globally $\eta$-transverse sequence
of sections provided by \cite{Donaldson:hypersurface}.  Pick 
$\epsilon \in (0,\eta)$, and let $t_D : M \to \LD$ denote the 
$\epsilon$-asymptotically holomorphic sequence provided by
Proposition~\ref{prop:Mohsen}, giving rise to the perturbed sections
$s_D' := s_D + t_D$ and zero-sets $Z_D := (s_D')^{-1}(0) \subset M$.  
Using Lemma~\ref{lemma:surjectivity}(b),
we may assume $s_D'$ is also $K$-asymptotically holomorphic and
$\eta$-transverse after making the substitutions
$K \mapsto K + \epsilon > 0$ and $\eta \mapsto \eta - \epsilon > 0$, and by
shrinking $\eta > 0$ further if 
necessary, Proposition~\ref{prop:Mohsen} also guarantees
$$
\Surj_{g_D}\left( \nabla s_D'(p)|_{\xi_p}\right) \ge \eta
$$
for all $p \in Z_D \cap V$.  This implies that for sufficiently large~$D$,
$Z_D \subset (M,\omega)$ is a symplectic submanifold 
and intersects both $V$ and the
distribution $\xi \subset TV$ transversely, hence the submanifold
$$
\Sigma_D := Z_D \cap V \subset V
$$
inherits a smooth oriented hyperplane bundle
$$
\xi_D := T Z_D \cap \xi \subset T\Sigma_D.
$$
Regarding $\xi_D$ as a real subbundle of the complex vector bundle
$(\xi|_{\Sigma_D},J_0)$, Lemma~\ref{lemma:surjectivity}(c) and
\eqref{eqn:conformal} now imply
$$
\Theta_g\left(\xi_D ; \xi|_{\Sigma_D},J_0\right) \le
\max_{p \in \Sigma_D} \frac{2 \| \dbar s_D'(p)|_{\xi_p} \|_{g_D}}{\Surj_{g_D}\left(\nabla s_D'(p)|_{\xi_p}\right)}
\le \frac{2K}{\eta \sqrt{D}} \to 0
$$
as $D \to \infty$.  Since $V$ is $J_0$-convex,
there exists a contact form $\alpha$ on $V$ such that $\xi = \ker\alpha$
and $d\alpha|_{\xi}$ is a symplectic vector bundle structure that tames~$J_0$.
Applying Lemma~\ref{lemma:Theta}, we therefore conclude from the above
that $(\xi_D,d\alpha)$ is a symplectic subbundle of $(\xi|_{\Sigma_D},d\alpha)$
for sufficiently large~$D$, implying that $\alpha|_{T\Sigma_D}$ is contact,
so $\Sigma_D \subset (V,\xi)$ is a contact submanifold.  Moreover,
the complex structure $J_0|_{\xi}$ along $\Sigma_D$ admits a $C^0$-small 
perturbation to a complex structure $J_1$ on $\xi$ along $\Sigma_D$ for which
$\xi_D$ is $J_1$-invariant.  Following the extension procedure of
\cite{CieliebakMohnke:transversality}*{\S 8}, $J_1$ can then
be extended to an almost complex
structure on $M$ that preserves $\xi$ along~$V$, preserves $TZ_D$ and is
$C^0$-close to $J_0$ for sufficiently large~$D$.  Note that having $J_1$
be $C^0$-close to $J_0$ implies that $J_1|_{\xi}$ is also tamed by
$d\alpha|_{\xi}$ without loss of generality, thus $V$ is $J_1$-convex.
\end{proof}

\section{The proof}
\label{sec:proof}

We now proceed to the proof of the main theorem.

Suppose $(M,\omega)$ is a closed and connected symplectic 
manifold with an almost complex structure $J$ such that either of the 
following conditions are satisfied:
\begin{itemize}
\item $(M,\omega)$ is semipositive and $J$ is $\omega$-tame;
\item $[\omega] \in H^2(M;\ZZ)$ and $J$ is $\omega$-compatible.
\end{itemize}
We will assume the Gromov-Witten invariants to be defined via the prescriptions
in \S\ref{sec:semipositive} or \S\ref{sec:CM} accordingly.
Suppose $V \subset M$ is a $J$-convex hypersurface.
Arguing by contradiction, we assume there is a nontrivial
Gromov-Witten invariant of the form
\begin{equation}
\label{eqn:GWuniruled}
\GW^{(M,\omega)}_{0,m,A}(\PD[V] \cup \alpha_1,\alpha_2,\ldots,\alpha_m ; \beta) \ne 0
\end{equation}
for some $m \ge 3$, $A \in H_2(M)$, $\alpha_1,\ldots,\alpha_m \in H^*(M;\QQ)$ and 
$\beta \in H_*(\overline{\mM}_{0,m};\QQ)$.  
The essential idea of the proof will be show that \eqref{eqn:GWuniruled} 
implies the existence of a pseudoholomorphic sphere that touches $V$ tangentially from 
the wrong side, thus contradicting pseudoconvexity.

\begin{remark}
\label{remark:unified}
In the following we will give a unified argument that applies to both the
semipositive and non-semipositive cases, referring as necessary to the
slightly different sets of definitions in \S\ref{sec:semipositive}
and~\S\ref{sec:CM}.  For the semipositive case, some statements would need to
be modified in obvious ways by removing all references to
$\beta \in H_*(\overline{\mM}_{0,m})$ and the forgetful map
(see Remark~\ref{remark:gravity}).
\end{remark}

We must now choose a perturbed almost complex structure $J_1$ that
is suitably adapted to the definition of the Gromov-Witten invariants.  In the
semipositive case, it suffices to set $J_1 = J$.  If $(M,\omega)$ is not
semipositive, then we have assumed $[\omega] \in H^2(M;\ZZ)$ and can therefore
find a sequence of Donaldson hypersurfaces $Z_D$ of large degrees $D \in \NN$
as described in \S\ref{sec:CM}.  By Proposition~\ref{prop:hypersurfaces},
after making the degree sufficiently large, we can find a smooth 
$\omega$-tame almost complex structure $J_1$ that is arbitrarily $C^0$-close
to $J$ while making $Z_D$ a $J_1$-holomorphic hypersurface and
$V$ simultaneously a $J_1$-convex hypersurface.  We shall treat $J_1$ as the \emph{reference}
almost complex structure used in Lemmas~\ref{lemma:sequenceSemipositive}
and~\ref{lemma:sequence}. 

Let $J'$ denote a generic domain-dependent or 
$\overline{\mM}_{\ell+1}$-dependent perturbation of $J_1$ as described in
\S\ref{sec:semipositive} or~\S\ref{sec:CM} respectively, 
giving rise to the moduli
space $\mM_{0,m}^A(M,J')$ of $J'$-holomorphic spheres homologous to~$A$,
with the associated evaluation/forgetful pseudocycle
$$
(\ev,\Phi) = (\ev_1,\ldots,\ev_m,\Phi) : \mM_{0,m}^A(M,J') \to M^m \times \overline{\mM}_{0,m}.
$$
In the non-semipositive case, we are assuming as in \S\ref{sec:CM} that
$J'$ matches $J_1$ near $Z_D$ and the elements of $\mM_{0,m}^A(M,J')$ have
extra marked points constrained to lie in $Z_D$ under evaluation, but these
details will play no role in what follows and we will therefore
suppress them in the notation.  The condition \eqref{eqn:GWuniruled} now means
$$
[(\ev,\Phi)] \cdot \Big( \left([V] \cdot \PD(\alpha_1)\right) \times
\PD(\alpha_2) \times \ldots \times \PD(\alpha_m) \times \beta \Big) \ne 0.
$$

\begin{lemma}
\label{lemma:lemmu}
There exists a smooth loop
$$
\ell : S^1 \to \mM_{0,m}^A(M,J')
$$
such that $(\ev_1 \circ \,\ell)_*[S^1] \cdot [V] \ne 0$.
\end{lemma}
\begin{proof}
We lose no generality by supposing that the classes $\alpha_1,\ldots,\alpha_m
\in H^*(M;\QQ)$ and $\beta \in H_*\left(\overline{\mM}_{0,m}\right)$ are each
homogeneous, i.e.~they have well-defined degrees.
By a theorem of Thom \cite{Thom:bordism}, there are rational numbers $c_0,\ldots,c_m \ne 0$ and smooth submanifolds
$\bar{\alpha}_1,\ldots,\bar{\alpha}_m \subset M$ and $\bar{\beta} \subset \overline{\mM}_{0,m}$ such that
\begin{equation*}
\begin{split}
c_0 [\bar{\beta}] &= \beta \in H_*(\overline{\mM}_{0,m};\QQ), \\
c_i [\bar{\alpha}_i] &= \PD(\alpha_i) \in H_*(M;\QQ) 
\quad \text{ for $i=1,\ldots,m$}.
\end{split}
\end{equation*}
We claim that after generic smooth perturbations of these submanifolds, we 
may assume the pseudocycle $(\ev,\Phi)$ is weakly transverse to
$\bar{\alpha}_1 \times \ldots \times \bar{\alpha}_m \times \bar{\beta}$
in the sense of \cite{McDuffSalamon:Jhol}*{Definition~6.5.10}. Indeed,
we can perturb $\bar{\alpha}_1$ such that $\ev_1$ is weakly transverse to
$\bar{\alpha}_1$, so by \cite{McDuffSalamon:Jhol}*{Lemma~6.5.14},
$$
\ev_2|_{\ev_1^{-1}(\bar{\alpha}_1)} : \ev_1^{-1}(\bar{\alpha}_1) \to M
$$
is a pseudocycle of dimension $\dim \mM_{0,m}^A(M,J') - \deg \alpha_1$.
After perturbing $\bar{\alpha}_2$, we may also assume this new pseudocycle is
weakly transverse to $\bar{\alpha}_2$, which means $(\ev_1,\ev_2)$ is now
weakly transverse to $\bar{\alpha}_1 \times \bar{\alpha}_2$.  Repeating this
procedure $m+1$ times proves the claim.  With this established, we can define
the \emph{constrained moduli space}
$$
\mM' := (\ev,\Phi)^{-1}(\bar{\alpha}_1 \times \ldots \times \bar{\alpha}_m
\times \bar{\beta}),
$$
so that $(\ev,\Phi)|_{\mM'}$ is a $1$-dimensional pseudocycle,
which means $\mM'$ is a \emph{compact} $1$-dimensional submanifold of
$\mM_{0,m}^A(M,J')$.  Now choose a generic smooth perturbation $V'$ of
$V \subset M$ such that
$$
\bar{\alpha}_1 \pitchfork V' \quad\text{ and }\quad
\ev_1|_{\mM'} \pitchfork V'.
$$
We then have
\begin{equation}
\begin{split}
\label{eqn:Poincare}
c_0 \ldots & c_m \Big( (\ev_1)_*[\mM'] \cdot [V] \Big) = \\
& [(\ev,\Phi)] \cdot \Big( \left( [V] \cdot \PD(\alpha_1)\right) \times
\PD(\alpha_2) \times \ldots \times \PD(\alpha_m) \times \beta\Big) \ne 0.
\end{split}
\end{equation}
Any connected component of $\mM'$ on which the above intersection number is 
nonzero is then a smooth loop with the stated property.
\end{proof}

In order to apply this lemma in proving the main result, 
we shall borrow an idea from \cite{AlbersBramhamWendl}.
Observe that by \eqref{eqn:GWuniruled}, $[V] \in H_*(M;\QQ)$ must be 
nontrivial, so~$V$ is nonseparating.  One can therefore construct a
connected infinite cover of~$M$, defined by cutting $M$ open along~$V$ to 
produce a cobordism with boundary $-V \sqcup V$, and then gluing together 
an infinite chain of copies $\{ M_n \}_{n \in \ZZ}$ of this cobordism.  
Denote for each $n \in \ZZ$ the boundary of the 
cobordism $M_n$ by
$$
\p M_n = -V_n^- \sqcup V_n^+,
$$
then each $V_n^\pm$ has a neighborhood in $M_n$ naturally identified with a 
suitable half-neighborhood of~$V$ in~$M$, and we use these identifications 
to glue $M_n$ to $M_{n+1}$ along $V_n^+ = V_{n+1}^-$.  This produces a 
smooth, connected and noncompact manifold (see Figure~\ref{fig:infiniteChain})
$$
\widetilde{M} = \bigcup_{n\in \ZZ} M_n, 
$$
which has a natural smooth covering projection
$$
\pi : \widetilde{M} \to M
$$
and is separated by infinitely many copies of the hypersurface~$V$, which we 
shall denote by
$$
V_n := V_n^+ \subset \widetilde{M}.
$$
Let
$$
\widetilde{J}_1 := \pi^*J_1
$$
denote the natural lift of the reference almost complex structure $J_1$
to the cover $\widetilde{M}$, for which the hypersurfaces 
$V_n$ are all $\widetilde{J}_1$-convex.

\begin{figure}
\psfrag{M}{$M$}
\psfrag{Mtilde}{$\widetilde{M}$}
\psfrag{V}{$V$}
\psfrag{V1}{$V_1$}
\psfrag{V0}{$V_0$}
\psfrag{V-1}{$V_{-1}$}
\psfrag{V-2}{$V_{-2}$}
\psfrag{M0}{$M_0$}
\psfrag{M1}{$M_1$}
\psfrag{M-1}{$M_{-1}$}
\psfrag{pi}{$\pi$}
\centering
\includegraphics{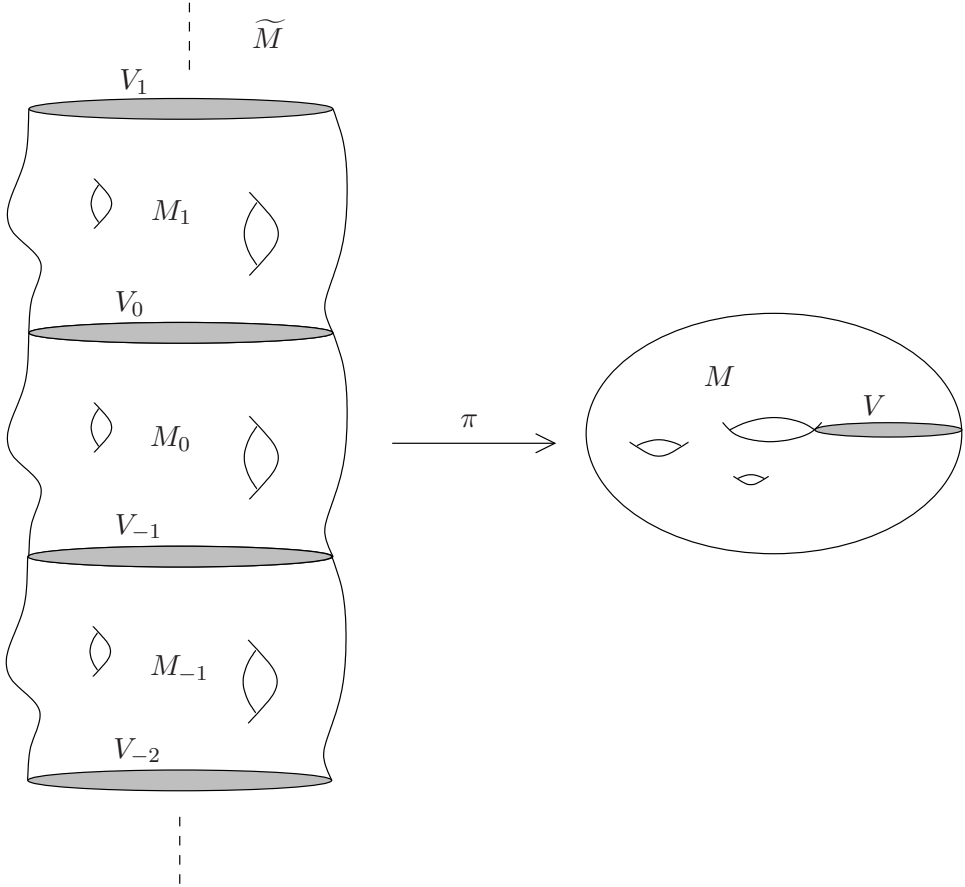}
\caption{The cover $\pi : \widetilde{M} \to M$ defined for a nonseparating hypersurface $V \subset M$.}
\label{fig:infiniteChain}
\end{figure}

By Lemma~\ref{lemma:sequenceSemipositive} or~\ref{lemma:sequence}, we
can find a sequence $J^k$ of generic structures for which 
Lemma~\ref{lemma:lemmu} holds with $J' : = J^k$, producing loops
$$
\ell_k : S^1 \to \mM_{0,m}^A(M,J^k)
\quad\text{ with }\quad
(\ev_1 \circ \,\ell_k)_*[S^1] \cdot [V] \ne 0 \text{ for all~$k$},
$$
and we may assume moreover that $J^k$ converges in $C^\infty$ as $k \to \infty$ 
to the domain-independent almost complex structure~$J_1$.
For each $k$ and each $\tau \in S^1$, $\ell_k(\tau) \in \mM_{0,m}^A(M,J^k)$ 
is an equivalence class of
spheres $u : S^2 \to M$ satisfying a domain-dependent Cauchy-Riemann
equation as in \eqref{eqn:CR}.  Since $S^2$ is simply connected, each of
the loops $\ell_k$ can be lifted to
$\widetilde{M}$ as a continuous family of holomorphic spheres 
$\{ u_\tau^k \}_{\tau \in \RR}$, and the nontrivial intersection 
of $\ev_1 \circ\, \ell_k$ with~$V$ implies that evaluation of $u_\tau^k$ 
at the first marked point traces a noncompact path in $\widetilde{M}$ 
intersecting $M_n$ for every $n \in \ZZ$.  It follows that for each~$k$, 
there exists a parameter value $\tau_*^k \in \RR$ for which the image of 
$u^k_{\tau_*^k}$ touches~$V_0$ but not the interior of~$M_1$.

We now have a sequence of curves $u^k := u^k_{\tau_*^k} 
\in \mM_{0,m}^A(M,J^k)$ which
admit lifts to $\widetilde{M}$ that touch~$V_0$ but not the interior of~$M_1$.
This is not yet a contradiction, because the Cauchy-Riemann equation
satisfied by each $u^k$ involves a domain-dependent almost complex structure.
As $k \to \infty$, however, Gromov compactness gives a subsequence of $u^k$ 
converging to a nodal $J_1$-holomorphic sphere, and at least one
component of this nodal curve lifts to a nontrivial 
$\widetilde{J}_1$-holomorphic sphere in $\widetilde{M}$
that touches $V_0$ tangentially from below.  Since $V_0$ is a
$\widetilde{J}_1$-convex hypersurface, 
this is a contradiction and thus concludes the proof.

\appendix

\section{The forgetful map is a pseudocycle}
\label{sec:forgetful}

The purpose of this appendix is to justify the statement, made in \S\ref{sec:CM},
that for suitably chosen data, the evaluation/forgetful map
$$
(\ev,\Phi) : \mM_{0,m+\ell}^A(M,J;Z_D) \to M^m \times \overline{\mM}_{0,m}
$$
as defined in the setting of Cieliebak and Mohnke 
\cite{CieliebakMohnke:transversality} is a pseudocycle, and its
rational cobordism class (after dividing by $\ell!$)
is independent of the choices.  This is proved in 
\cite{CieliebakMohnke:transversality} for $\ev : \mM_{0,m+\ell}^A(M,J;Z_D) \to M^m$,
without accounting for the forgetful map, though the arguments necessary
for proving the more general statement are almost already present in
\cite{CieliebakMohnke:transversality}, so we shall merely sketch the
necessary modifications.

In the following, we will often refer to holomorphic curves that carry
distinct sets of \defin{ordinary} and \defin{extra} marked points; for
curves in the space $\mM_{0,m+\ell}^A(M,J;Z_D)$, this means the first~$m$ and
last~$\ell$ marked points respectively.  Recall that
the forgetful map $\Phi : \overline{\mM}_{0,m+\ell}(M,J;Z_D) \to
\overline{\mM}_{0,m}$ is defined by forgetting not only the map into~$M$
but also the extra marked points, and then stabilizing.

\begin{remark}
\label{remark:notCompact}
Although $\Phi$ maps the top stratum
$\mM_{0,m+\ell}^A(M,J;Z_D)$ into the top stratum $\mM_{0,m}$ of
$\overline{\mM}_{0,m}$, it will not generally define a pseudocycle
$\mM_{0,m+\ell}^A(M,J;Z_D) \to \mM_{0,m}$, mainly because $\mM_{0,m}$
itself is not compact.
\end{remark}

We assume as in \S\ref{sec:CM} that $J_0$ is a compatible almost
complex structure on the closed and connected $2n$-dimensional symplectic manifold
$(M,\omega)$, and $Z_D \subset M$ is a nearly 
$J_0$-holomorphic Donaldson hypersurface of large degree~$D \in \NN$.
If $D$ is sufficiently large and $J \in \jJ_{\ell + 1}$ is
chosen appropriately (e.g.~it must be $C^0$-close to $J_0$ and match a
reference domain-independent structure $J_1$ near $Z_D$, whose restriction to
$Z_D$ is generic), then \cite{CieliebakMohnke:transversality} shows that 
the natural compactification $\overline{\mM}_{0,m+\ell}^A(M,J;Z_D)$ of
$\mM_{0,m+\ell}^A(M,J;Z_D)$ consists of strata
$\mM_T^{\{A_\alpha\}}(M, J  ; Z_D)$ modelled on weighted
$(m+\ell)$-labelled trees $(T,\{A_\alpha\})$ that are \defin{$\ell$-stable}, 
i.e.~they are stable even after removing the $m$ ordinary (but keeping the
$\ell$ extra) marked points.  Moreover, none of the nonconstant components of 
such nodal curves are 
contained in~$Z_D$.  The pseudocycle property for $(\ev,\Phi)$ is based on the
observation that on any stratum $\mM_T^{\{A_\alpha\}}(M, J  ; Z_D) \subset
\overline{\mM}_{0,m+\ell}^A(M,J;Z_D)$ for which $T$ has more than one vertex,
the restriction of $(\ev,\Phi)$ factors as a composition
\begin{equation}
\label{eqn:factorization}
\mM_T^{\{A_\alpha\}}(M,J ; Z_D) \to \mM_{T'}^{\{A_\alpha\}}(M,J ; Z_D) \to M^m \times
\overline{\mM}_{0,m},
\end{equation}
where the space in the middle is a smooth manifold that either has dimension
at most $\dim \mM_{0,m+\ell}^A(M,J ; Z_D) - 2$ or factors through another
manifold that does.  The reason we need this factorization instead of just
considering $\mM_T^{\{A_\alpha\}}(M,J ; Z_D)$ itself is that the latter
sometimes has artificially large dimension, due to the presence of multiple
extra marked points in the same constant component.  But since these extra
marked points play no role in defining the evaluation and forgetful map,
we can fix this problem by removing them, which leads to the factorization
above.  The remainder of this appendix will be concerned with
the definition and essential properties of 
$\mM_{T'}^{\{A_\alpha\}}(M,J ; Z_D)$.

As in \cite{CieliebakMohnke:transversality}, we will use the term
\defin{ghost tree} to mean a maximal subtree $T''$ of a
weighted tree $(T,\{A_\alpha\})$ with the property that $A_\alpha=0$
for all $\alpha \in T''$.  Similarly, a \defin{ghost bubble} on
a nodal $J$-holomorphic curve $[(\mathbf{z},\mathbf{u})] \in
\mM_{T}^{\{A_\alpha\}}(M,J;Z_D)$ is the constant holomorphic curve 
obtained by restricting $\mathbf{u}$ to any component 
$S_\alpha \subset \Sigma_{\mathbf{z}}$ with
$A_\alpha=0$.  We shall define the manifold $\mM_{T'}^{\{A_\alpha\}}(M,J;Z_D)$ 
roughly as the space of nodal curves that one obtains from
elements of $\mM_T^{\{A_\alpha\}}(M,J;Z_D)$ by forgetting all but one of
the extra marked
points on each ghost tree and stabilizing as necessary,
but keeping all other information, including the conformal structures on 
the ghost bubbles with their ordinary marked points.  This can be defined
more formally as follows.  
Suppose $\ell' \le \ell$ is the number of extra marked points on
vertices $\alpha \in T$ with $A_\alpha \ne 0$ plus the number of
ghost trees in $T$ that have at least one extra marked point.
Then we associate to $T$ a stable
$(m+\ell')$-labelled tree $T'$ via the following procedure:
\begin{enumerate}
\item On each ghost tree in $T$, keep all ordinary marked points and 
the first extra marked point (if any)
but remove all other extra marked points;
\item Stabilize by removing any vertices that now have
fewer than 3 special points and adjusting neighboring
edges accordingly.  (Note that since $T$ is stable, this step
can only affect vertices $\alpha \in T$ with $A_\alpha=0$.)
\end{enumerate}
By Lemma~\ref{lemma:coherence}, any 
coherent almost complex structure $J \in \jJ_{\ell+1}$ determines
for every nodal curve $\mathbf{z}$ modelled on $T$ a 
$\Sigma_{\mathbf{z}}$-dependent almost complex structure $J_{\mathbf{z}}$ 
whose restriction to each component $S_\alpha \subset \Sigma_{\mathbf{z}}$
depends only on the special points on~$S_\alpha$.  It follows that if
$\mathbf{z}$ is modelled on $T'$, then $J$ uniquely determines a domain
dependent almost complex structure on any component 
$S_\alpha \subset \Sigma_{\mathbf{z}}$ with $A_\alpha \ne 0$
(cf.~the discussion preceding Corollary~5.9 in 
\cite{CieliebakMohnke:transversality}).  We can extend this to 
a $\Sigma_{\mathbf{z}}$-dependent almost complex structure
$$
J_{\mathbf{z}} \in \jJ_{T'}
$$
by setting $J_{\mathbf{z}}|_{S_\alpha}$ for each $\alpha \in T'$ with
$A_\alpha=0$ to match the fixed domain-independent reference almost complex
structure~$J_1$.  In this way, we can 
speak of \emph{nodal $J$-holomorphic maps}
$(\mathbf{z},\mathbf{u})$ modelled on the weighted $(m+\ell')$-labelled
tree $(T',\{A_\alpha\})$; note that the definition of $J_{\mathbf{z}}$
on components $S_\alpha$ with $A_\alpha=0$ plays no role here since 
$\mathbf{u}$ is necessarily constant on such components.
Denote by $\widetilde{\mM}_{T'}^{\{A_\alpha\}}(M,J;Z_D)$ the
space of such maps for which the $\ell'$ extra marked points are all
mapped into~$Z_D$, and denote its quotient by the group of biholomorphic
isomorphisms by
$$
\mM_{T'}^{\{A_\alpha\}}(M,J;Z_D) := \widetilde{\mM}_{T'}^{\{A_\alpha\}}(M,J;Z_D) / \sim.
$$
There is a natural projection
$$
\mM_{T}^{\{A_\alpha\}}(M,J;Z_D) \to \mM_{T'}^{\{A_\alpha\}}(M,J;Z_D)
$$
defined by forgetting $\ell - \ell'$ of the extra marked points and then
collapsing constant components as necessary in order to stabilize the domain.
Since all the ordinary marked points are retained in this process, the
factorization \eqref{eqn:factorization} of $(\ev,\Phi)$ is well defined.
The pseudocycle property now \emph{mostly} follows from the following
lemma, whose proof is exactly the same as
\cite{CieliebakMohnke:transversality}*{Lemma~5.6, Prop.~5.7 and Cor.~5.8}.

\begin{lemma}
\label{lemma:codimension2}
For generic $J$, if $e(T')$ denotes the number of edges in the tree~$T'$, then 
$\mM_{T'}^{\{A_\alpha\}}(M,J;Z_D)$ is a smooth manifold with
\begin{equation*}
\begin{split}
\dim \mM_{T'}^{\{A_\alpha\}}(M,J;Z_D) &= 2(n-3) + 2 c_1(A) + 2m - 2 e(T') \\
&= \dim \mM_{0,m+\ell}^A(M,J;Z_D) - 2 e(T'). 
\end{split}
\end{equation*}  \qed
\end{lemma}

We must still deal with the possibility
that $T$ has
more than one vertex but $T'$ has only one,
in which case $\mM_{T'}^{\{A_\alpha\}}(M,J;Z_D)$ can be regarded as a space of
smooth (non-nodal) 
curves $\mM_{0,m+\ell'}^A(M,J;Z_D)$ constrained to send their $\ell'$
extra marked points into~$Z_D$.\footnote{Since technically $J$ belongs to
$\jJ_{\ell+1}$ and not $\jJ_{\ell'+1}$, the definition of $J$-holomorphicity
for curves in $\mM_{0,m+\ell'}^A(M,J;Z_D)$ is a bit subtle and must be understood
in the same sense as the preceding discussion of $\mM_{T'}^{\{A_\alpha\}}(M,J;Z_D)$.}
This space
has dimension equal to that of $\mM_{0,m+\ell}^A(M,J;Z_D)$, but
we claim that for generic $J$, if $T$ has more than one vertex,
then curves in $\mM_{0,m+\ell'}^A(M,J;Z_D)$ that
arise in this way from elements of $\mM_{T}^{\{A_\alpha\}}(M,J;Z_D)$
lie in a subset of codimension at least~$2$.  The crucial point here is 
that such a curve will never belong to the open subset
$$
\mM_{0,m+\ell'}^{A,*}(M,J;Z_D) \subset \mM_{0,m+\ell'}^A(M,J;Z_D)
$$
consisting of curves whose intersections with $Z_D$ at the $\ell'$ extra
marked points are all transverse, and for generic $J$, 
\cite{CieliebakMohnke:transversality}*{\S 6} shows that the complement of
this subset is a finite union of smooth submanifolds having
dimension at most $\dim \mM_{0,m+\ell'}^A(M,J;Z_D) - 2$.
To see that curves in $\mM_{0,m+\ell'}^{A,*}(M,J;Z_D)$ are excluded,
observe that the curves in question arise precisely in situations
where removing the relevant extra marked points from
ghost bubbles in $T$ makes all of them unstable---in particular, 
$(T,\{A_\alpha\})$ must in this case consist of the following:
\begin{itemize}
\item A unique vertex $\alpha_0$ that has all $m$ of the ordinary marked points
and $A_{\alpha_0} = A \ne 0$;
\item One or more ghost trees that each have no ordinary marked points but 
at least two of the extra marked points.
\end{itemize}
The resulting curve in $\mM_{0,m+\ell'}^A(M,J;Z_D)$ is not contained in~$Z_D$ but
has $\ell'$ marked points at which it must intersect~$Z_D$, and if
all of these $\ell'$ intersections are transverse, then the fact that
$A \cdot [Z_D] = \ell > \ell'$ implies there must be additional intersections
separate from the extra marked points. But since these curves are assumed
to arise from objects in the closure of $\mM_{0,m+\ell}^A(M,J;Z_D)$, the latter
implies (via positivity of intersections) 
the existence of curves in $\mM_{0,m+\ell}^A(M,J;Z_D)$ that have
intersections with $Z_D$ outside their extra marked points, and that is
impossible.  This proves:

\begin{lemma}
\label{lemma:incidence}
For generic $J$, 
if $T$ has more than one vertex and $T'$ has only one, then
the restriction of $(\ev,\Phi)$ to $\mM_T^{\{A_\alpha\}}(M,J;Z_D)$ factors as
\begin{equation*}
\begin{split}
\mM_T^{\{A_\alpha\}}(M,J;Z_D) &\to \mM_{0,m+\ell'}^A(M,J;Z_D) \setminus
\mM^{A,*}_{0,m+\ell'}(M,J;Z_D) \\
& \to M^m \times \overline{\mM}_{0,m},
\end{split}
\end{equation*}
where the space in the middle is a finite union of manifolds having
dimension at most $\dim \mM_{0,m+\ell}^A(M,J;Z_D) - 2$.  \qed
\end{lemma}

It follows from Lemmas~\ref{lemma:codimension2} and~\ref{lemma:incidence}
that for generic~$J$, $(\ev,\Phi)$ is a pseudocycle as claimed.
Using these same factorizations, one can similarly adapt the proof of
\cite{CieliebakMohnke:transversality}*{Theorem~1.3} to show that the
rational pseudocycle defined by $\frac{1}{\ell!} (\ev,\Phi)$ is independent
of the choices $(J_0,Z_D,J)$ up to rational cobordism.
This involves defining corresponding moduli spaces for $1$-parameter families
of data, as well as moduli spaces of curves with two sets of extra marked points
constrained by two Donaldson hypersurfaces of differing degrees---the idea in
each case is to factor $(\ev,\Phi)$ as above through moduli spaces in which each 
ghost tree carries at most one extra marked point.  Such moduli spaces always
have small enough dimension to establish the pseudocycle condition.

\end{document}